\documentclass{amsart}

\newtheorem{theorem}{Theorem}[section]
\newtheorem{lemma}[theorem]{Lemma}
\newtheorem{proposition}[theorem]{Proposition}
\newtheorem{corollary}[theorem]{Corollary}

\theoremstyle{definition}
\newtheorem{definition}[theorem]{Definition}
\newtheorem{example}[theorem]{Example}

\newtheorem{remark}[theorem]{Remark}

\usepackage{amscd,amssymb}

\begin{document}

\title[Almost prime ideals in noncommutative rings]
{Almost prime ideals in noncommutative rings}

\author[Alaa Abouhalaka, {\c S}ehmus F{\i}nd{\i}k]
{Alaa Abouhalaka, {\c S}ehmus F{\i}nd{\i}k}

\address{Department of Mathematics,
\c{C}ukurova University, 01330 Balcal\i,
 Adana, Turkey}
\email{alaa1aclids@gmail.com}
\email{sfindik@cu.edu.tr}
\thanks
{}

\subjclass[2010]{13A15; 17A01.}
\keywords{Almost prime ideals; noncommutative rings.}

\begin{abstract}
A proper ideal $P$ of a commutative ring with identity is an almost prime ideal if $ab \in P{\setminus}P^2$ implies $a \in P$ or $b \in P$.
In this paper we define almost prime ideals of a noncommutative ring,
and provide some equivalent definitions. We also examine some cases such that all right ideals of a noncommutative ring are almost prime right ideals.  
\end{abstract}

\maketitle

\section{Introduction}

We initiate by some basic definitions.
Let $R$ be a noncommutative ring. A right ideal $P$ of $R$ is called prime
if $AB\subseteq P$  implies $A\subseteq P$ or $B\subseteq P$ for all right ideals $A,B$ of $R$.
If the ring $R$ is with identity, then the previous definition is equivalent to the following. A right ideal $P$ of $R$ is called prime if
$aRb\subseteq P$ implies $a\in P$ or $b \in P$ for all $a,b\in R$.
A right ideal $P$ of $R$ is called weakly prime if  $ 0 \neq AB \subseteq P$ implies  $A \subseteq P$ or $B \subseteq P$, for all right ideals $A,B$ of $R$ \cite{HPT}.
A right ideal $P$ of $R$ is called idempotent if $P^2=P$. 

\subsection{Generalizations of prime ideals in commutative rings}

Almost prime ideals appear by studying of unique factorization in noetherian domains introduced by
Bhatwdekar and Sharma \cite{BS} in 2005. They proved in Theorem 2.15 that almost prime ideals are precisely primes in regular domains. 

On the other hand Bataineh \cite{B} (2006) handled the generalizations of prime ideals in his PhD thesis widely,
by giving more properties and theorems about almost prime ideals in commutative rings with identity.
He characterized noetherian rings in which every proper ideal is a product of almost prime ideals,
and proved that there are several equivalent conditions for an ideal to be almost prime in commutative rings
in Theorem 2.8 as follows.

\begin{theorem} Let $P$ be an ideal of a commutative ring $R$ with identity. Then the followings are equivalent. 

$(1)$ $P$ is almost prime.

$(2)$ If $a\in R{\setminus}P$, then $P:\{a\}=P\cup(P^2:\{a\})$.

$(3)$ If $a\in R{\setminus}P$, then $P:\{a\}=P$ or $P:\{a\}=P^2:\{a\}$.

$(4)$  If $AB \subseteq P$ and $AB\not\subseteq P^2$, then $A \subseteq P$ or $B \subseteq P$, for all ideals $A,B$ of $R$.

\end{theorem}

\noindent Note that here $I:J=\left\{x\in R \mid Jx\subseteq I\right\}$ for subsets $I,J$ of $R$.

Also, Theorem 2.20 states that if $\langle a\rangle$ is an almost prime ideal of $R$, for a nonzero nonunit element $a$,
such that $0 : \langle a\rangle\subseteq \langle a\rangle$, then $\langle a\rangle$ is prime. His another result (Theorem 2.19) gives  a criterion on the forms of
almost prime ideals of cartesian product of rings.
\begin{theorem}

Let $R$ and $S$ be two commutative rings with identity. Then an ideal of $R \times S$ is almost prime if and only if it has one of the following three forms.

$(1)$ $I\times S$, where $I$ is an almost prime ideal of $R$.

$(2)$ $R \times J$, where $J$ is an almost prime ideal of $S$. 

$(3)$ $I \times J$ , where $I$ and $J$ are idempotent ideals of $R$ and $S$, respectively. 
\end{theorem}
\noindent He clarified the relationship between almost prime and weakly prime ideals (Theorem 2.13):
An ideal $P$ of a commutative ring $R$ is an almost prime ideal if and only if $P/P^2$ is a weakly prime ideal of $R/P^2$.

Another result (Corollary 2.8) by \cite{BS} states that in a quasilocal domain $R$, with maximal ideal $M$,
for an ideal $I$, with $M^2\subseteq I \subseteq M$; $I$ is almost prime if and only if $M^2 = I^2$. 
Later, it turned out to be true in any quasilocal ring, see Proposition 2.23 of \cite{B}.
It is also known by \cite{B} that in commutative ring $R$  every proper ideal is almost prime if and only if $R$ is
either von Neumann regular or quasilocal ring with maximal ideal $M$ such that $M^2 = 0$.

Recall that $P$ is a weakly prime ideal of a commutative ring $R$,
if $0\neq ab\in P$ implies $a\in P$ or $b\in P$ for all $a,b\in R$.
Anderson and Smith \cite{AS} studied weakly prime ideals for a commutative ring with identity in 2003.
They proved that if $P$ is weakly prime but not prime, then $P^2 = 0$ (Theorem 1).  
Also, they \cite{AS} showed in Theorem 3 that for an ideal $P$ of a commutative ring  $R$ with identity, the following statements are equivalent.

$(1)$ $P$ is a weakly prime ideal.

$(2)$ If $a\in R{\setminus}P$, then $P:\{a\}=P\cup(0:\{a\})$.

$(3)$ If $a\in R{\setminus}P$, then $P:\{a\}=P$ or $P:\{a\}=0:\{a\}$.

$(4)$  $0 \neq AB \subseteq P$ implies  $A \subseteq P$ or $B \subseteq P$, for all ideals $A,B$ of $R$.

\noindent The previous theorem shows us how similar prime and weakly prime ideals are.
It is well known that an ideal $P$ of a commutative ring  $R$ with identity is prime equivalent that
$P:\{a\}=P$ for all $a\in R{\setminus}P$, which in turn, is equivalent to the condition $AB \subseteq P$ implies
$A\subseteq P$ or $B \subseteq P$, for all ideals $A,B$ of $R$ (see e.g., Theorem 2.3, \cite{B}).

It is known by Theorem 7 of \cite{AS} that if $R= R_1\times R_2$ for commutative rings $R_1$, $R_2$ with identies,
then the ideal $P$ is a weakly prime ideal of $R$ implies $P =0$ or $P$ is prime. 

\subsection{Generalizations of prime ideals in noncommutative rings}
Hirano et al. \cite{HPT} extended the notion of weakly prime ideals in rings, nonnecessarily commutative
or with identity in 2010. They defined weakly prime (right) ideals and proved that
for an ideal $P$ in a ring $R$ with identity, the following statements are equivalent.

$(1)$ $P$ is a weakly prime ideal.

$(2)$ If $J, K$ are right ideals of $R$ such that $0\neq J.K \subseteq P$, then $J\subseteq P$ or $K\subseteq P$.

$(3)$ If  $a,b\in R$, such that $0\neq aRb\subseteq P$ then $a\in P$ or $b\in P$.

\noindent They \cite{HPT} also studied the structure of rings, not necessarily commutative, in which all (right) ideals are weakly prime,
as an analogue of the earlier work of Blair and Tsutsui \cite{BT}, who studied rings in which every ideal is prime.

Most recently in 2020, Groenewald \cite{Gr}
gave more properties of weakly prime ideals in noncommutative rings, and introduced the notion of the weakly prime radical of an ideal.
Theorem 1.14 states that for an ideal $P$ of any ring $R$, the following statements are equivalent.

$(1)$ $P$ is a weakly prime ideal.

$(2)$ If $a\in R{\setminus}P$, then $P:\langle a)=P\cup(0:\langle a))$.

$(3)$ If $a\in R{\setminus}P$, then $P:\langle a)=P$ or $P:\langle a)=0:\langle a)$.

\noindent Here the set $\langle a)=Ra+\mathbb Za$ is the left ideal of $R$ generated by $a$.
Similarly one may denote by $(a\rangle=aR+\mathbb Za$ the right ideal of $R$ generated by $a$.
Clearly if $R$ is a ring with identity, then $(a\rangle=aR$, $\langle a)=Ra$,
and in this case $\langle a\rangle=RaR$ is the principle ideal of $R$ generated by $a$.
The theorem above is a generalization of Theorem 3 of \cite{AS} given by Anderson and Smith. 

\subsection{The objective of the paper}
In this paper, we define almost prime (right) ideals in a noncommutative ring $R$ as follows.
\begin{definition}\label{basic}
A proper (right) ideal $P$ of $R$ is called an almost prime (right) ideal
if $AB\subseteq P$ and $AB\not\subseteq P^2$ implies either $A \subseteq P$ or $B \subseteq P$, for all (right) ideals $A$, $B$ of $R$.
\end{definition}
\noindent A quick observation gives that our definition and the concept of almost prime ideals in commutative rings are equivalent.
However, the definitions differ in noncommutative rings.
We show in Theorem \ref{theorem28} that $P$ is an almost prime ideal in a noncommutative ring $R$ with identity
if and only if $aRb\subseteq P$ and $aRb\not\subseteq P^2$ implies either $a \in P$ or $b \in P$  for all $a,b\in R$.

Note that every (weakly) prime right ideal is an almost prime right ideal.
Hence the concept of almost prime right ideal is a generalization of prime ideals.
However any almost prime right ideal does not need to be a prime right ideal, see Example \ref{example} $(i)$-$(ii)$.

We give several equivalent conditions for an ideal to be almost prime. We also show the analogy between almost prime right ideals, prime right ideals, and weakly prime right ideals.
Theorem \ref{weakly} of the present paper is a generalization of Theorem 2.13 of \cite{B} in a noncommutative setting. 
We also investigate images and the inverse images of almost prime right ideals under ring homomorphisms.

Finally, we define fully almost prime right rings.
Example \ref{example} $(ii)$ is an example of its existence. Also we show some cases in which we get a fully almost prime right ring.

Throughout this paper all rings are associative, noncommutative, and without identity unless stated otherwise, and by ideal we mean a proper two sided ideal.

\section{Almost prime right ideals}

It is clear by Definition \ref{basic}
that every prime (weakly prime, idempotent) right ideal is an almost prime right ideal. However the inverse is not true in general,
as indicated in the next example.

\begin{example}\label{example}
  
$(i)$  $0=\left\{ 0\right\}$ is an almost prime ideal of any ring. A ring whose zero ideal is prime (semiprime) is called prime ring (semiprime ring). Hence in this sense every ring is an almost prime ring.

$(ii)$ Let $R=\left\{0,a,b,c\right\}$ be the noncommutative ring with binary operations
\begin{center}
\begin{tabular}{|c|c|c|c|c|}
\hline
 +&0&$a$&$b$&$c$\\
\hline
0&$0$&$a$&$b$&$c$\\
\hline
$a$&$a$&$0$&$c$&$b$\\
\hline
$b$&$b$&$c$&$0$&$a$\\
\hline
$c$&$c$&$b$&$a$&$0$\\
\hline
\end{tabular}
\quad
\begin{tabular}{|c|c|c|c|c|}
\hline
$.$ &$0$&$a$&$b$&$c$\\
\hline
0&0&0&0&0\\
\hline
$a$&0&$a$&$a$&0\\
\hline
$b$&0&$b$&$b$&0\\
\hline
$c$&0&$c$&$c$&$0$\\
\hline
\end{tabular}
\end{center}
\noindent It is straightforward to show that $I=\left\{0,b\right\}$, $J=\left\{0,c\right\}$, $P=\left\{0,a\right\}$ are all right ideals of $R$.
Moreover, $IJ=0\subseteq P$,  $I\not\subseteq P$,  and $J\not\subseteq P$. Hence the right ideal $P$ is not a prime right ideal. On the otherhand,
$P$ is an almost prime right ideal as a consequence of the fact that $P^2=P$. 

$(iii)$ An application of our previous example  is the subring $R$ of $M_2(\mathbb Z_2)$, where  
\[
R=\{0,e_{11}+e_{12},e_{21}+e_{22},e_{11}+e_{12}+e_{21}+e_{22}\}.
\]
The right ideals of $R$ are $I=\{0,e_{21}+e_{22}\}$, $J=\{0,e_{11}+e_{12}+e_{21}+e_{22}\}$, and $P=\{0,e_{11}+e_{12}\}$.
Then $P$ is an almost prime right ideal of $R$ which is not prime. Notice that the ideals $I,J$ are also almost prime right ideals.
Hence all the right ideals of the ring $R$ are almost prime right ideals. 

$(iv)$ In the ring $R=\{ae_{11}+be_{12}+ce_{22}\mid a,b,c\in\mathbb R\}$ with identity, 
the right ideal $P=\{ae_{22}\mid a\in\mathbb R\}$
is an almost prime right ideal, but it is not a prime right ideal, since
$e_{12}(ae_{11}+be_{12}+ce_{22})2e_{12}=0$, and thus $0=(e_{12})R(2e_{12})\subseteq P$; however, $e_{12},2e_{12}\notin P$.
\end{example}

\begin{proposition} Let $R$ be a ring with identity, and $P$ be an ideal of $R$. Then the following statements are equivalent.

$(1)$ $P$ is an almost prime right ideal.

$(2)$  $P$ is an almost prime  ideal.

\end{proposition}

\begin{proof}  $(1)\Rightarrow(2)$ Trivial by definition.

$(2)\Rightarrow(1)$ Suppose that  $AB\subseteq P$, and $AB\not\subseteq P^2$, for right ideals $A$ and $B$ of $R$.
Since $R$ is with identity, then $AR=A$, and
\[
(RA)(RB) = RAB \subseteq RP=P
\]
for ideals $RA$ and $RB$. On the other hand, if $(RA)(RB) \subseteq P^2$, then
\[
AB\subseteq RAB= (RA)(RB) \subseteq P^2
\]
is a contradiction. Thus $(RA)(RB) \not\subseteq P^2$, and by $(2)$ we have
either $A\subseteq RA\subseteq P$ or  $B\subseteq RB\subseteq P$.
\end{proof}

In the following we give some similarities between almost prime right ideal, weakly prime right ideal and prime right ideal. 
Initially, let us remind the next definition.

\begin{definition}  Let $R$ be a ring and $I,J$ right ideals of $R$. The right ideal $I:J$ (known as colon right ideal) of $R$
is defined as $I:J=\left\{x \in R \mid Jx\subseteq I\right\}$. Similarly we define the set $(I:J)^*=\left\{x \in R \mid  xJ\subseteq I\right\}$.
Recall that in case of $I,J$ being ideals, so are $I:J$ and $(I:J)^*$.   
\end{definition} 

The following proposition is well known.

\begin{proposition}\label{AB} For the right ideals $A$, $B$, $P$ of any ring $R$, if $P\subseteq A\cup B$ then either $P\subseteq A$ or $P\subseteq B$.
In particular, if $P= A\cup B$ then either $P= A$ or $P= B$.

\end{proposition}

Now we provide several equivalent conditions of being an almost prime ideal.  

\begin{theorem}\label{theorem28}
Let $R$ be ring with identity, and $P$ be an ideal of $R$ then the following statements are equivalent:

$(1)$ $P$ is an almost prime ideal.

$(2)$ If $(a\rangle(b\rangle\subseteq P$, $(a\rangle(b\rangle\not\subseteq P^2$, then either $a \in P$ or $b\in P$, where $a,b\in R$. 

$(3)$ If $aRb\subseteq P$, $aRb\not\subseteq P^2$, then either $a\in P$ or $b\in P$, where $a,b\in R$.  

$(4)$ $P:\langle a\rangle$= $P\cup (P^2:\langle a\rangle)$ and $(P:\langle a\rangle)^*=P\cup (P^2:\langle a\rangle)^*$ for all $a\in R{\setminus}P$.

$(5)$ Either $P:\langle a\rangle= P$ or $P:\langle a\rangle=P^2:\langle a\rangle$, and
either $(P:\langle a\rangle)^*= P$ or  $(P:\langle a\rangle)^*=(P^2:\langle a\rangle)^*$  for all $a\in R{\setminus}P$.
\end{theorem}

\begin{proof} 
 $(1)\Rightarrow(2)$ Let $a,b\in R$ such that $(a\rangle (b\rangle\subseteq P$, $(a\rangle (b\rangle\not\subseteq P^2$. This implies that
\[
(RaR)(RbR)=R(a\rangle R(b\rangle=R(a\rangle (b\rangle\subseteq RP=P.
\]
Assume that $(RaR)(RbR)\subseteq P^2$. Then
\[
(a\rangle (b\rangle\subseteq R(a\rangle (b\rangle=(RaR)(RbR)\subseteq P^2
\]
contradicts with $(a\rangle (b\rangle\not\subseteq P^2$. Hence, $(RaR)(RbR)\not\subseteq P^2$.
Thus, by (1) either $RaR\subseteq P$ or $RbR\subseteq P$, which implies that either $a \in P$ or $b \in P$.

 $(2)\Rightarrow(3)$ Let $a,b\in R$ such that $aRb\subseteq P$, $aRb\not\subseteq P^2$. Then
\[
(a\rangle(b\rangle=aRbR\subseteq PR=P.
\]
Assuming that $(a\rangle(b\rangle\subseteq P^2$, we get
\[
aRb\subseteq aRbR=(a\rangle(b\rangle\subseteq P^2,
\]
which is a contradiction. Hence, $(a\rangle(b\rangle\not\subseteq P^2$, and by $(2)$ either $a \in P$ or $b \in P$.

$(3)\Rightarrow(4)$ Let $a\in R{\setminus}P$ and $b\in P:\langle a\rangle$. Then, by definition of colon ideals we get $aRb\subseteq\langle a\rangle b\subseteq P$.
If $aRb\not\subseteq P^2$, then by $(3)$ we have $b\in P$, since $a\in R{\setminus}P$. 
If $aRb\subseteq P^2$, then $\langle a\rangle b=RaRb\subseteq RP^2=P^2$. Hence, $b\in P^2:\langle a\rangle$. 
Consequently, $P:\langle a\rangle\subset P\cup (P^2:\langle a\rangle)$.

Now let $b\in P\cup (P^2:\langle a\rangle)$. If $b\in P$, then $\langle a\rangle b\subseteq \langle a\rangle P\subseteq P$, and thus $b\in P:\langle a\rangle$. 
If $b\in P^2:\langle a\rangle$, then $\langle a\rangle b\subseteq P^2\subseteq P$, which implies that $b\in P:\langle a\rangle$. Therefore,
$P\cup (P^2:\langle a\rangle)\subseteq P:\langle a\rangle$, and hence $P:\langle a\rangle=P\cup (P^2:\langle a\rangle)$.
One may prove the second equality similarly.

$(4)\Rightarrow(5)$ The proof is a consequence of Proposition \ref{AB}.

$(5)\Rightarrow(1)$ Let $A$ and $B$ be two ideals of $R$ such that $AB\subseteq P$.
Assume that $A\not\subseteq P$ and $B\not\subseteq P$. Hence, we have to show that $AB\subseteq P^2$.

Let $a\in A{\setminus}P$. Then, we have $\langle a\rangle B\subseteq AB\subseteq P$, which implies by definition
of colon ideal that
$B\subseteq P:\langle a\rangle$. Now by $(5)$ we get $B\subseteq P^2:\langle a\rangle=P:\langle a\rangle$,
since $B\not\subseteq P$. Therefore, $aB \subseteq\langle a \rangle B \subseteq P^2$.
Consequently, $(A{\setminus}P)B \subseteq P^2$.

Let $b\in B{\setminus}P$. Then, $A\langle b \rangle\subseteq AB\subseteq P$, and so $A\subseteq(P:\langle b  \rangle)^*$.
Now by $(5)$ we get $A\subseteq(P^2:\langle b \rangle)^*=(P:\langle b  \rangle)^*$, because $A\not\subseteq P$.
Thus, $Ab\subseteq A\langle b \rangle\subseteq P^2$, which implies that $A(B{\setminus}P)\subseteq P^2$.
Finally, we get that
\begin{align}
AB&=(A{\setminus}P)B+(A\cap P)(B{\setminus}P)+(A\cap P)(B\cap P)\nonumber\\
&\subseteq(A{\setminus}P)B+ A(B{\setminus}P)+(A\cap P)(B\cap P)\subseteq P^2,\nonumber
\end{align}
which completes the proof.
\end{proof}
Note that an almost prime right ideal does not have to be a prime right ideal or a weakly prime right ideal. 
In the next two theorems, we examine some cases such that the statements above are equivalent.

\begin{theorem}
Let $R$ be ring with identity, and $P$ be a right ideal of $R$ such that $(P^2:P)\subseteq P$. Then the following statements are equivalent.

$(1)$ $P$ is a prime right ideal.

$(2)$ $P$ is an almost prime right ideal. 

\end{theorem}

\begin{proof}  $(1)\Rightarrow(2)$ Trivial by definition.

$(2)\Rightarrow(1)$ Assume that the almost prime right ideal $P$ is not a prime right ideal.
Then, there exist right ideals $A$, $B$ of $R$ such that $AB\subseteq P$ with $A\not\subseteq P$ and $B\not\subseteq P$.
Hence, by $(2)$ we get $AB\subseteq P^2$.
Let $a\in A{\setminus}P$ and $b\in B{\setminus}P$. 
Then,
\[
((a\rangle+P)(b\rangle=(a\rangle.(b\rangle+P(b\rangle\subseteq AB+P(b\rangle\subseteq P.
\] 
If $((a\rangle+P)(b\rangle\subseteq P^2$, then $P(b\rangle\subseteq P^2$. This implies
that $Pb\subseteq P^2$, and thus $b\in( P^2:P)\subseteq P$. This contradicts with $b\in B{\setminus}P$.
If $((a\rangle+P)(b\rangle\not\subseteq P^2$, then by $(2)$ we get that 
either $(a\rangle+P\subseteq P$ or $(b\rangle\in P$, which implies $a\in P$ or $b\in P$, respectively. Contradiction.

\end{proof}

\begin{theorem}\label{basis}  Let $R$ be ring and $P$ be a right ideal of $R$ such that $P^2=0$. Then the following statements are equivalent.

$(1)$ $P$ is a weakly prime right ideal.

$(2)$ $P$ is an almost prime right ideal. 

\end{theorem}
\begin{proof}   $(1)\Rightarrow(2)$ Follows from definition.

$(2)\Rightarrow(1)$ Suppose that  $A$ and $B$ are right ideals of $R$ such that $0\neq AB\subseteq P$.
Then $AB \not\subseteq P^2=0$. Thus, by $(2)$ we are done. 
\end{proof}

The next result is a consequence of Theorem \ref{basis}.

\begin{corollary}\label{corollarybasis}
Let $R$ be ring such that $R^2=0$, and let $P$ be a right ideal of $R$. Then the following statements are equivalent.

$(1)$ $P$ is a weakly prime right ideal.

$(2)$ $P$ is an almost prime right ideal. 

\end{corollary}

Recall that  (see e.g., \cite{L}) a nonzero right ideal $P$ of a ring $R$ is called a minimal right ideal,
if for any right ideal $I$ of $R$ with $0\subseteq I\subseteq P$, either $I$ = $0$ or $I = P$.
The next lemma is known as {\it Brauer's lemma}.

\begin{lemma}\label{Brauer}\cite{L}
Any minimal right ideal $P$ of a ring $R$ with identity satisfies $P^2=0$ or $P=(e\rangle$ for some idempotent $e$ of $R$.
\end{lemma}

It is well known that every right ideal generated by an idempotent element is an idempotent ideal.
Hence, combining Lemma \ref{Brauer} and Theorem \ref{basis} one obtains the following corollary.

\begin{corollary}\label{corollary}
Let $P$ be a minimal right  ideal of a ring $R$  with identity. If $P$ is an almost prime right ideal but not an idempotent ideal, then $P$ is a weakly prime right ideal. 
\end{corollary}

Note that in Example \ref{example} $(ii)$, $J=\left\{0,c\right\}$ is an almost prime minimal right ideal, and $J^2=0$.
Hence, $J$ is a weakly prime right ideal by Corollary \ref{corollary}.

Bataineh \cite{B} proved that an ideal $P$ is an almost prime ideal if and only if
$P/P^2$ is a weakly prime ideal of $R/P^2$ in a commutative ring $R$.
In the next theorem, we show that the same holds for almost prime right ideals and weakly prime right ideals in noncommutative rings.

\begin{theorem}\label{weakly} Let $R$ be a ring and $P$ be an ideal of $R$. Then, the following statements are equivalent.

$(1)$ $P$ is an almost prime right ideal of $R$.

$(2)$  $P/P^2$  is weakly prime right ideal of $R/P^2$. 

\end{theorem}

\begin{proof} $(1)\Rightarrow(2)$ Suppose that $\bar I, \bar J$ are right ideals of $R/P^2$ such that
\[
\bar 0\neq\bar I\bar J\subseteq\bar P=P/P^2.
\]
Thus, there exist right ideals $I\supseteq P^2$ and $J\supseteq P^2$ of $R$ such that $\bar I=I/P^2$ and $\bar J$=$J/P^2$.
Therefore, $P^2/P^2\neq(IJ+P^2)/P^2 \subseteq P/P^2$, thus $P^2\neq IJ\subseteq P$.
Now by $(1)$ we have that either $I\subseteq P$ or $J\subseteq P$, since $IJ\not\subseteq P^2$.
This yields that either $\bar I\subseteq\bar P$ or $\bar J\subseteq\bar P$.

$(2)\Rightarrow(1)$ Suppose that $I,J$ are right ideals of $R$ such that $IJ\subseteq P$ and $IJ\not\subseteq P^2$.
Then, $\bar I$=$(I+P^2 )/P^2$, $\bar J=(J+P^2 )/P^2$ are right ideals of $R/P^2$.
Moreover,
\[
\bar I\bar J=(IJ+IP^2+P^2J+P^4+P^2 )/P^2\subseteq P/P^2=\bar P,
\]
and $\bar I\bar J\not\subseteq\bar P^2$. Thus, $\bar 0\neq\bar I\bar J\subseteq\bar P$ and by $(2)$, either $\bar I\subseteq\bar P$ or $\bar J\subseteq\bar P$.
Consequently, $I\subseteq P$ or $ J\subseteq P$.
\end{proof}

Let $f:R\to S$ be a ring homomorphism for commutative rings $R,S$, and $P$ be an ideal of $R$ such that $\text{\rm ker}f\subseteq P$.
It is well known that $P$ is a prime ideal of $R$ if and only if $f(P)$ is a prime ideal of $S$.
The next theorem is a noncommutative analogue of the problem for almost prime right ideals.

\begin{theorem}\label{fully1}
Let $f:R\to S$ be a ring epimorphism, and $P$ be an almost prime right ideal of $R$ such that
$\text{\rm ker}f\subseteq P$. Then $f(P)$ is an almost prime right ideal of $S$.
\end{theorem}

\begin{proof} 
Suppose that $A_2B_2\subseteq f(P)$ and $A_2B_2\not\subseteq (f(P))^2$ for right ideals $A_2,B_2$ of $S$.
Then, $\text{\rm ker}f\subseteq f^{-1}(A_2 )=A_1$ and $\text{\rm ker}f\subseteq f^{-1}(B_2 )=B_1$ are right ideals of $R$.
Hence, $f(A_1 )=A_2$ and $f(B_1 )=B_2$, since $f$ is an epimorphism. Then, we have that
\[
f(A_1B_1 )=A_2B_2\subseteq f(P),
\]
and
\[
f(A_1B_1 )\not\subseteq(f(P))^2=f(P^2).
\]
Thus $A_1B_1\subseteq f^{-1}(f(A_1B_1))\subseteq f^{-1}(f(P))=P$ and $A_1B_1\not\subseteq P^2$. 
Now by assumption, either $A_1\subseteq P$ or $B_1\subseteq P$; i.e., either $A_2\subseteq f(P)$ or $B_2\subseteq f(P)$.
\end{proof}

\begin{corollary}
Let $f:R\to S$ be a ring epimorphism, and $B$ be a right ideal of $S$ such that $f^{-1} (B)$ is an almost prime right ideal of $R$.
Then, $B$ is an almost prime right ideal of $S$.
\end{corollary}
\begin{proof}
Since the inverse image of any right ideal of $S$ is a right ideal of $R$ containing $\text{\rm ker}f$, the proof follows from Theorem \ref{fully1}.
\end{proof}

\begin{theorem}\label{fully2}
Let $f:R\to S$ be a ring epimorphism, and $P$ be a right ideal of $R$ such that $\text{\rm ker}f\subseteq P^2$.
If $f(P)$ is an almost prime right ideal of $S$, then $P$ is an almost prime right ideal of $R$.
\end{theorem}

\begin{proof}
Suppose that $A_1B_1\subseteq P$ and $A_1B_1\not\subseteq P^2$ for right ideals $A_1,B_1$ of $R$.
Then, $f(A_1)f(B_1)=f(A_1B_1)\subseteq f(P)$.
If $f(A_1B_1 )\subseteq f(P^2)$, then
\[
A_1B_1\subseteq f^{-1}(f(A_1B_1 ))\subseteq f^{-1}(f(P^2 ))=P^2,
\]
which is a contradiction. Hence, $f(A_1)f(B_1)=f(A_1B_1)\not\subseteq(f(P))^2$.
Since $f(P)$ is an almost prime right ideal of $S$, then either $f(A_1)\subseteq f(P)$
or $f(B_1)\subseteq f(P)$. Consequently, either $A_1\subseteq f^{-1}(f(A_1 ))\subseteq f^{-1}(f(P))=P$
or $B_1\subseteq P$.
\end{proof}

\begin{corollary}
Let $f:R\to S$ be a ring epimorphism, and $B$ be an almost prime right ideal of $S$ such that $\text{\rm ker}f\subseteq(f^{-1} (B))^2$.
Then, $f^{-1}(B)$ is an almost prime right ideal of $R$.
\end{corollary}

\begin{proof}
Let $P=f^{-1}(B)$. Then, $P$ is an almost prime right ideal of $R$ by Theorem \ref{fully2},
since $\text{\rm ker}f\subseteq P^2$ and $f(P)=f(f^{-1}(B))=B$ is an almost prime right ideal of $S$.
\end{proof}

Theorem \ref{fully1} and Theorem \ref{fully2} show us how the almost prime characteristic is transferred
from one ring to another when they are linked by a homomorphism.
The following theorem is the transition of almost prime property to quotient rings.

\begin{theorem}\label{fully3}
Let $R$ be a ring, $I$ be an ideal of $R$. Let $P$ be a right ideal $P$ of $R$ such that $I\subseteq P$.
If $P$ is an almost prime right ideal of $R$ then $P/I$ is an almost prime right ideal of $R/I$. 
\end{theorem}

\begin{proof}
Suppose that $\bar A\bar B\subseteq\bar P=P/I$ and $\bar A\bar B\not\subseteq \bar P^2$ for right ideals $\bar A,\bar B$ in $R/I$. 
Assume that $\bar A=A/I$ and $\bar B=B/I$ for some right ideals $A\supseteq I$ and $B\supseteq I$.
Then, $(AB+I)/I\subseteq P/I$ and $(AB+I)/I\not\subseteq(P^2+I)/I$,
which implies that $AB\subseteq P$ and $AB\not\subseteq P^2$. 
So, either $A\subseteq P$ or $B\subseteq P$, and thus $\bar A\subseteq\bar P$ or $\bar B \subseteq\bar P$.
\end{proof}

The converse of Theorem \ref{fully3} is not true in general. For instance, let $P=I$ be
an ideal which is not an almost prime right ideal. However, the zero ideal $\bar 0=I/I$ is an almost prime right ideal of $R/I$ .

\section{Fully almost prime right rings}

\begin{definition} 
A ring in which every right ideal is an almost prime right ideal is called a fully almost prime right ring.
\end{definition}

Note that every fully prime right ring (fully weakly prime right ring, fully idempotent right ring) is a fully almost prime right ring.
The ring $R$ in Example \ref{example} $(ii)$ is a fully almost prime right ring.
The next result is a consequence of Theorem \ref{basis}.
\begin{corollary}\label{reslt}
Let $R$ be a ring such that $P^2=0$ for every right ideal $P$ of $R$. Then, the following statement are equivalent.

$(1)$ $R$ is fully almost prime right ring.

$(2)$ $R$ is fully weakly prime right ring.
\end{corollary}

\begin{remark}
Corollay \ref{corollarybasis} suggests that the assumption of Corollary \ref{reslt} can be replaced by $R^2=0$.
\end{remark}

\begin{theorem}
Let $f:R\to S$ be a ring epimorphism. If $R$ is a fully almost prime right ring, so is $S$.
\end{theorem}

\begin{proof}
Let $P$ be a right ideal of $S$. Then $f^{-1}(P)\supseteq\text{\rm ker}f$ is an almost prime right ideal of $R$.
Then, by Theorem \ref{fully1} we get that $f(f^{-1} (P))=P$ is an almost prime right ideal of $S$.
\end{proof}

\begin{theorem} Let $f:R\to S$ be a ring epimorphism such that ${\rm ker}f\subseteq I^2$,
for any right ideal $I$ of $R$. If $S$ is a fully almost prime right ring, so is $R$.
\end{theorem}

\begin{proof} Let  $P$ be a right ideal of $R$. Then, $f(P)$ is an almost prime right ideal of the fully almost prime right ring $S$. 
Hence, by Theorem \ref{fully2} we get that $P$ is an almost prime right ideal of $R$.
\end{proof}

\begin{theorem}
Let $R$ be a ring, and $I$ be an ideal of $R$. If $R$ is a fully almost prime right ring, so is $R/I$.
\end{theorem}

\begin{proof}
Suppose $\bar P$ is a right ideal of $R/I$. Then, there exist a right ideal $P\supseteq I$ of $R$
such that $\bar P=P/I$. Clearly, $P$ is an almost prime right ideal of $R$.
Hence, by Theorem \ref{fully3} $\bar P$ is an almost prime right ideal of $R/I$ .
\end{proof}


\begin{thebibliography}{ABC}

\bibitem[B]{B}
M. Bataineh,
Generalization of prime ideals,
Phd thesis, The University of Iowa, 2006.

\bibitem[BS]{BS}
M. S. Bhatwadekar, P. K. Sharma,
Unique factorization and brith of almost prime,
Comm. Algebra, 33 (2005) 1, 43-49.

\bibitem[BT]{BT}
W. D. Blair, H. Tsutsui,
Fully prime rings,
Comm. Algebra, 22 (1994) 13, 5389–5400. 

\bibitem[Gr]{Gr}
N. Groenewald,
Weakly prime and weakly completely prime ideals of noncommutative rings,
Int. Electron. J. Algebra, 28 (2020) 43-60.  

\bibitem[HPT]{HPT}
Y. Hirano, E. Poon, H. Tsutsui,
On rings in which every ideal is weakly prime,
Bull. Korean Math. Soc., 47 (2010) 5, 1077-1087.   

\bibitem[L]{L}
T. Y. Lam,
A first course in noncommutative rings,
Graduate Texts in Mathematics, 131, Springer-Verlag, New York, 2001. 

\bibitem[AS]{AS}
D. D. Anderson and E. Smith,
Weakly prime ideals,
Houston J. Math., 29 (2003) 4, 831-840.

\end{thebibliography}
\end{document}